\documentclass[10pt,psamsfonts]{article}
\usepackage{amsmath}
\usepackage{amssymb}
\usepackage{amsthm}
\usepackage{fancybox}
\usepackage[all]{xy}
\usepackage{amscd}

\newtheorem{theorem}{Theorem}[section]

\newtheorem{corollary}[theorem]{Corollary}
\newtheorem{proposition}[theorem]{Proposition}

\newtheorem{remark}[theorem]{Remark}
\newtheorem{example}[theorem]{Example}

\newtheorem{definition}[theorem]{Definition}

\newcommand{\m}{\frak{m}}
\newcommand{\n}{\frak{n}}

\def\cocoa
{\mbox{\rm C\kern-.13em o\kern-.07
em C\kern-.13em o\kern-.15em A}}

\def\dim{{\rm dim}}
\def\depth{{\rm depth}}

\def\Syz{{\rm Syz}}
\def\Gin{{\rm Gin}}

\newcounter{itemitemcounter}

\newcounter{itemcounter}

\title{  \bf   Minimal free resolution  of  a finitely generated module over a  regular local ring
 }
\author{\large    Maria Evelina  Rossi  and  Leila Sharifan  }

\date { }

\begin{document}

\maketitle

\begin{center}{\it Department of Mathematics,
University of Genoa,\\ Via Dodecaneso 35, 16146 Genoa, Italy\\
rossim@dima.unige.it}
\end{center}

\begin{center}{\it Faculty of Mathematics and  Computer Science, Amirkabir University of Technology\\
424, Hafez Ave., P. O. Box 15875-4413, Tehran, Iran   \\
leila-sharifan@aut.ac.ir}
\end{center}

\begin{abstract}   Numerical invariants of a minimal free resolution of a module $M$ over a  regular local  ring $(R,\n)$  can  be studied by taking advantage of  the rich literature on the graded case. The key is to fix suitable $\n$-stable filtrations  ${\mathbb M} $  of $M $ and to compare the  Betti numbers of  $M$  with  those  of the   associated graded module $ gr_{\mathbb M}(M).  $ This  approach has the advantage that  the same module $M$ can be detected by using different filtrations on it.
It  provides  interesting  upper bounds for the Betti numbers and we  study the  modules for which the extremal values are attained.   Among others,  the Koszul modules have this behavior. As a consequence of the main result, we extend   some results by Aramova,  Conca, Herzog and Hibi  on the rigidity of the resolution of standard graded algebras to the local  setting.


\end{abstract}

\maketitle

 \section*{Introduction}

Consider a local ring  $(R, \n)$ and let $M$ be a finitely generated $R$-module.  In the literature, starting from classical results by   Northcott,  Abhyankar,  Matlis and   Sally,    several  authors  detected   basic numerical characters  of   the module $ M $    by means of  the Hilbert Function of $M $  arising from the standard $\n$-adic filtration or,  more in general,  from $\n$-stable  filtrations   (see  \cite{RV}  for an extensive overview).    Deeper  information can   be achieved  from the numerical invariants of a minimal free resolution of $M.$
It is a classical tool  to equip  $M$ with a suitable filtration  and to
get information on  $M$ from the graded free resolution of the corresponding associated graded module.  This approach  has the advantage to benefit from the rich literature concerning the graded cases   to return the   information to   the local ones.  In particular, the main goal of the paper is to  compare the numerical invariants of a local ring $(A, \m)$  with those  of the     associated graded ring with respect to the $\m$-adic filtration on $A,$   that is   the standard graded algebra  $gr_{\m}(A):=\oplus_{t\ge 0}\m^t/\m^{t+1}. $  This is a graded object   which  corresponds to a relevant geometric construction that encodes several  information   on  $A.$  In fact if $A$ is the localization at the origin of the coordinate ring of an affine variety $V$ passing through $0$, then the associated graded ring $gr_{\m}(A)$ is the coordinate ring of the {\it{tangent cone}}  of $V. $ Often, we may assume that $A=R/I $ where $R$ is a regular local ring. In this case  $gr_{\m}(A) \simeq P/I^* $ where $I^*$ is a homogeneous  ideal of a polynomial ring $P$ generated by the initial forms (w.r.t. the $\n$-adic filtration) of the elements of $I.$ In this setting,  the theory of the Gr\"obner bases and  the leading term  ideals   offers interesting  results on the extremal values of the Betti numbers (see for example \cite{AHH}, \cite{CHH}, \cite{C}).

\vskip 2mm

 It is the  reason why we are mainly interested in studying finitely generated modules $M$ over a regular local ring $(R, \n). $  In this case  $gr_{ \n}(R)=\oplus_{i\geq 0} \n^i/\n^{i+1} $  is a polynomial ring, say $P, $ and the associated graded module  with respect to the any $\n$-stable  filtration  of $M$  is  a finitely generated $P$-module.

\vskip 2mm

In general,  the     problem is how to lift the information from  the associated graded module with respect to an $\n$-stable  filtration  to the original
module $M. $  If   the associated graded module with respect to the $\n$-adic filtration $gr_{ \n}(M) =\oplus_{i\ge 0} \n^i
M/\n^{i+1} M $ has a   linear resolution (as a $gr_{ \n}(R)$-module), then  it is easy to see that the Betti
numbers of $M$ and $gr_{ \n}(M) $ coincide.    In this case, the module
$M$ is said to have a {\it{ linear resolution}} in the terminology
of Herzog, Simis and Vasconcelos (see \cite{HSV}) or of Sega (see \cite{Se}),
equivalently $M$ is {\it{ Koszul}} in the terminology of Herzog and
Iyengar (see \cite{HI}). Koszul modules are important  examples  in our investigation. In this paper, we prefer to say that $M$ is Koszul because to have linear resolution is misleading in the graded case.
We recall that a  local ring $(A,\m) $ is said Koszul if  the residue field $k$  is Koszul as an $A$-module, that is the graded $k$-algebra $gr_{\m}(A)= \oplus_{i\ge 0}  \m^i  /\m^{i+1}  $ is Koszul in the classical meaning introduced by Priddy (see \cite{HI}, Remark 1.10). By combining results by  Herzog, Iyengar (see \cite{HI}, Proposition 1.5, Definition 1.7 )  and  R\"omer (see \cite{R}, Theorem 3.2.8 ) or Martinez and Zacharia (see \cite{MZ}, Theorems 2.4, 2.5 and 3.4), we can conclude that  Koszul graded modules and componentwise linear modules coincide. Given these considerations, in this paper  we will speak of Koszul modules in the case of modules over a local ring and of componentwise linear modules or,  indifferently,  of Koszul modules in the graded case. 
\vskip 3mm

Recently,  many papers have been written extending classical results of  the theory of the associated graded ring    with respect to the $\n$-adic filtration   to the more general case of a stable  (or good)  filtration $ \mathbb{M }=\{M_i\}_{i\ge 0}$ on a finitely generated module $M  $  over a local ring $(R,\n).$  The associated graded module  $gr_{\mathbb{M}}(M)=\oplus_i  M_i/M_{i+1} $   has a natural structure as a  finitely generated  $G$-module where $G=gr_{\n}(R)=\oplus_{i\geq 0}
\n^i/\n^{i+1}. $       We are interested to
compare   the  free resolution  of $M$ as an $R$-module and the free
resolution of $gr_{ \mathbb{M}}(M) $  as  a $ G$-module.

An important  starting point of our investigation is a result due to Robbiano (see \cite{Rob} and also   \cite{HRV}, \cite{Se}) which says that from a minimal $G$-free resolution of  $gr_{ \mathbb{M}}(M) $ we can build up an  $R$-free resolution on $M$ which is  not  necessarily  minimal. Hence, for the Betti numbers of $M$ and  $ gr_{ \mathbb{M}}(M) $ one has
$$  \beta_i(M)  \le \beta_i(gr_{ \mathbb{M}}(M) )  $$
for every $i \ge 0.$  In particular, the free resolution of $M $  is minimal if and only if $\beta_i(gr_{ \mathbb{M}}(M) ) = \beta_i(M) $ for every $i\ge 0.$

According   to   \cite{HRV},  a  finitely generated module  $M $  is called of   {\it{ homogeneous type}}  with respect to an  $\n$-stable   filtration $\mathbb{M}  , $    if $  \beta_i(M) = \beta_i(gr_{ \mathbb{M}}(M) )  $  for every $i \ge 0. $ If $M$ is of homogeneous type,  then in particular depth $M$ = depth  $gr_{ \mathbb{M}}(M). $ Notice that a  Koszul module is a module of homogeneous type with respect to the $\n$-adic filtration,  conversely there are modules of homogeneous type which are not Koszul.

Under the assumption that $R$ is a regular local ring, the main   result of this  paper  (Theorem \ref{M}) says   that if the minimal number of generators of $M$ (as an $R$-module) and the minimal number of generators of $ gr_{ \mathbb{M}}(M)$ (as  a $G$-module) coincide,   then $M$ is of homogeneous type with respect to $ \mathbb{M},$   provided that  $ gr_{ \mathbb{M}}(M)$ is a componentwise linear module or equivalently $ gr_{ \mathbb{M}}(M)$ is  a Koszul module. Moreover,   Theorem \ref{lpd} says  that  if $ gr_{ \mathbb{M}}(M)$ is Koszul for some $\n$-stable filtration $\mathbb{M}, $ then $M$ is Koszul,   provided the minimal number of generators of $M$  and  $ gr_{ \mathbb{M}}(M)$  coincide. It is an interesting question to ask whether the converse holds.
\vskip 2mm

The first application deals  with the classical   case  of the associated graded ring of a local ring  $A=R/I $ ($(R, \n) $  is a regular local ring).   In this case, we will   apply  Theorem \ref{M}   in the case $ M=I $
   equipped with the $\n$-stable filtration $ \mathbb {M}=\{ I \cap \n^i  \}.  $  The point is that $I^* =gr_{\mathbb {M}}(I). $

 As a consequence    we prove   that if $I$ is minimally generated by an  $\n$-standard base and $I^* $ (the ideal generated by the initial forms of the elements of $I$) is componentwise linear, then   the numerical invariants  of a minimal free resolution of $A$ and those of $gr_{\m}(A) $ coincide (see Corollary \ref{I*}).   In particular, under these assumptions,  depth $A=$ depth $gr_{\m}(A).$ We point out that,  if this is the case, $I$ is a Koszul module by Theorem \ref{lpd}.   

It is worth  remarking that we can not delete the condition which $I^*$ is componentwise linear, in fact in general $ A $ and $gr_{\m}(A) $ don't have the same Betti numbers    even if the minimal number of generators of $I$ and $ I^* $ coincide, that is $I$ is minimally generated by an  $\n$-standard base of $I. $ For example this is the case if we consider  the defining ideal $I$  of the semigroup ring $ A=k[[t^{19}, t^{26}, t^{34}, t^{40}]]  $  (see \cite[Example, Section 3]{HRV}. Notice that   both $A$ and $ gr_{\m}(A)  $   are Cohen-Macaulay.

On the analogy with the graded case and by taking advantage of  it, we can rephrase the result in terms of suitable monomial ideals attached to the ideal $I. $ If we define    $\Gin(I) := \Gin (I^*)   $ the generic initial ideal of $I^*$ w.r.t. the reverlex order, then we may deduce from   the homogeneous  case   (see \cite{AHH, CHH, C}),       $$\beta_i(I) \le \beta_i(Gin(I)) $$ for every $i\ge 0 $
and,  as a consequence of Theorem \ref{M},     the   equality holds   if and only if $\beta_0(I) = \beta_0(Gin(I)) $ or equivalently $I$ is minimally generated by an $\n$-standard base and $I^*$ is componentwise linear (see Corollary \ref{I*}).  A similar result can be also presented in terms of $Lex(I) $ which is the unique lex segment ideal in $P$ such that $R/I$ and $P/Lex(I)$ have  the same Hilbert function. Taking advantage from  the graded case, Corollary \ref{LEX} is an  extension to the local case of a result by  Herzog and  Hibi.

 In \cite{CHH},    Conca,   Herzog and   Hibi  proved   an upper bound for the Betti numbers  of  the local ring $A $ in terms of the so-called generic annihilators of $A;  $ Corollary \ref{alfa} shows  that, under the assumption of the main result,  the extremal Betti numbers are achieved.

Theorem  \ref{lpd}  has another interesting consequence.
Assume    $ I \subseteq \n^2$ is a non  zero  ideal of $R.   $   It is known that the  Symmetric algebra   of the maximal
 ideal $\m$ of $A=R/I $ is $ \ S_A(\m) \simeq \oplus_{i \ge 0} \n^i/I \n^{i-1}.  $ If we consider the  $\n$-adic
  filtration on the ideal $I, $ then the associated graded module $gr_{\n}(I)= \oplus_{i \ge 0}  I \n^i/I \n^{i+1} $
   with respect to the filtration $\mathbb{M}=\{ I \n^i  \} $ sits inside  $S_A(\m) $ as a graded submodule, via the canonical embedding
 $gr_{\n}(I) (-2) \to   S_A(\m).  $   Herzog, Rossi and Valla in \cite[Theorem 2.13]{HRV} showed   that we can deduce the homological properties of the Symmetric algebra by  taking  advantage from this
  comparison.  By using this result and the fact that the ideals under investigation are Koszul,  Theorem 3.8 extends a recent result by Herzog, Restuccia and Rinaldo \cite[Theorem 3.9] {HRR}  on the depth of the Symmetric algebra.    \\

{\bf{ Acknowledgment}}: This work was done while the second author was visiting University of Genoa, Italy.  She is thankful to the university for the hospitality.
She also thanks the Ministry of science, research and technology of Iran for the financial support.

The authors wish to thank A. Conca and T. R\"omer for many stimulating discussions in connection with this paper. Special thanks to the Reviewer for such a careful reading of the paper
and  the useful suggestions.

\section{ Preliminaries on filtered modules}

Throughout the paper $(R, \n)$ is a  local ring
and $M$ is  a finitely generated $R$-module.  We say, according to the notation in \cite{RV}, that a filtration of submodules $\mathbb{M} = \{ M_n\}_{n\ge 0}  $ on $M$ is   an  $\n$-filtration if  $\n M_n \subseteq M_{n+1} $ for every $n \ge 0, $ and a   stable (or good) $\n$-filtration if $\n M_n =  M_{n+1} $ for all sufficiently large $n.$  In the following a {\it{ filtered module}} $M$ will be always an $R$-module equipped with a stable  $\n$-filtration $\mathbb{M}.$

If    $\mathbb{M}=\{M_j\}$ is an  $\n$-filtration of $M$, define
$$gr_{\mathbb{M}}(M)=\bigoplus_{j\ge 0}(M_j/M_{j+1})$$ which is a graded
$gr_{\n}(R)$-module in a natural way. It is called the {\bf associated graded module} to the filtration $\mathbb{M}.$

To avoid triviality, we assume that $gr_{\mathbb{M}}(M)$ is not zero or equivalently $M \not = 0.$
 If $N$ is a submodule of $M,$ by Artin-Rees Lemma, the sequence $\{N\cap M_j\ | \ j\ge 0\}$ is a good $\n$-filtration of $N$.  Since  \begin{equation}\label{N}  (N\cap M_j)/N\cap M_{j+1})\simeq (N\cap M_j+M_{j+1})/M_{j+1} \end{equation}   $gr_{\mathbb{M}}(N)$ is a graded submodule of $gr_{\mathbb{M}}(M) $ denoted by $N^*.$

 If $m \in M\setminus\{0\}, $ we denote by $\nu_{\mathbb{M}}(m) $ the largest integer $p$ such that  $ m \in M_p $ (the so-called valuation of $m$ with respect to $\mathbb{M}) $ and we denote by $m^* $ or $gr_{\mathbb{M}}(m) $ the residue class of $m$ in $M_p/M_{p+1} $ where $p= \nu_{\mathbb{M}}(m).  $ If $m=0, $ we set   $\nu_{\mathbb{M}}(m)= + \infty. $

   Using (\ref{N}), it is clear that  $gr_{\mathbb{M}}(N)$ is generated by the elements $x^* $ with $x \in N,$ we write $$ gr_{\mathbb{M}}(N) =< x^* \   : \ x \in N>.$$
   On the other hand  it is clear that $\{(N+M_j)/N \ | \ j\ge 0\}$ is a good $\n$-filtration of $M/N$ which we denote  by $\mathbb{M}/N.$ These graded modules are related by the graded isomorphism $$gr_{\mathbb{M}/N}(M/N)\simeq gr_{\mathbb{M}}(M)/gr_{\mathbb{M}}(N).$$
   \vskip 3mm
   For completeness we collect, in this section, a part of the well known  results  concerning the homomorphisms of filtered modules.
   \vskip 2mm
   \begin{definition} If $M$ and $N$ are filtered $R$-modules and $f : M \to N $ is an $R$-homomorphism, $f$ is said to be a {\it{homomorphism of filtered modules}}  if $ f(M_p) \subseteq N_p $ for every $p \ge 0 $ and $f$ is said {\it { strict }} if $f(M_p) = f(M) \cap N_p $  for every $p \ge 0. $
   \end{definition}
The morphism of filtered modules $f : M \to N $ clearly induces a
morphism of graded $gr_{\n}(R)$-modules $$gr(f) : gr_{\mathbb{M}}(M)
\to  gr_{\mathbb{N}}(N).$$ It is clear that $ gr ( \dot ) $ is a
functor from the category of the filtered $R$-modules into the
category of the graded $gr_{\n}(R)$-modules.  Furthermore we have a
canonical embedding $ (Ker f)^* \to Ker( gr(f)).$

  \vskip 3mm
 \begin{proposition} \label{RoV}{\rm{(see \cite{RoV})}} Let $ {\bf{F}} :    M \overset{g}  {\to} N \overset{f}  {\to} Q  $  be a complex of filtered modules and $$ gr({\bf{F}}) : gr_{\mathbb{M}}(M) \overset{gr(g)}{\to} gr_{\mathbb{N}}(N)
 \overset{gr(f)}{\to} gr_{\mathbb{Q}}(Q) $$ be the induced complex of graded $gr_{\n}(R)$-modules. Then $ gr({\bf{F}}) $ is exact if and only if $ {\bf{F}} $ is exact and $f$ and $g$ are strict morphisms.

 \end{proposition}

 As a consequence of the above result we get that a morphism $f : M \to N $ of filtered modules is strict if and only if the canonical embedding $ (Ker f)^* \to Ker (gr(f)) $ is an isomorphism.


 \begin{definition} {\label{L}} Let $L= \oplus_{i=1}^s R e_i $ be a free $R$-module of rank $s $ and $\nu_1, \dots, \nu_s $ be integers. We define the filtration $\mathbb{L} = \{L_p : p \in {\bf{Z}} \} $ on $L$ as follows
 $$L_p := \oplus_{i=1}^s \n^{p- \nu_i} e_i = \{(a_1, \dots , a_s) : a_i \in \n^{p- \nu_i} \}.$$
 We denote the filtered free $R$-module $L$ by $ \oplus_{i=1}^s R(-\nu_i) $ and we call it  {\it{ special filtration}} on $L.$
 \end{definition}

 So when we write $L=  \oplus_{i=1}^s R(-\nu_i) $ it means that we consider the free module $L$ of rank $s $ with the special  filtration  defined above. It is clear that $\mathbb{L} $ is an $\n$-stable filtration. A filtered free $R$-module $L=  \oplus_{i=1}^s R(-\nu_i) $  is a free module with canonical basis $(e_1, \dots, e_s)$ such that $\nu_{\mathbb{L}}(e_i)=  \nu_i.$ It is obvious that $gr_{\mathbb{L}}(L) = \oplus_{p}  L_p/L_{p+1} $ is isomorphic as  a $gr_{\n}(R)$-module to $\oplus_{i=1}^s gr_{\n}(R) (-\nu_i) = \oplus_{i=1}^s G(-\nu_i) $  where  for short $G=gr_{\n}(R).  $

The canonical basis $(gr_{\mathbb{L}}(e_1), \dots, gr_{\mathbb{L}}(e_s)) $ of $gr_{\mathbb{L}}(L) $ will be simply denoted by  $(e_1, \dots, e_s).$ Note that $R$ with the $\n$-adic filtration is the filtered module $R(0).$
 \vskip 2mm  If  $({\bf{F.}}, \delta .) $ is a complex of finitely generated free $R$-modules, a special filtration on  {\bf{F.}}  is a special filtration on each $F_i $ that makes $({\bf{F.}}, \delta .) $ a filtered complex (complex of filtered modules).  Our goal is to consider  special  filtrations on  an  $R$-free resolution of a filtered module $M. $
 We recall that over local rings, each finitely generated module has a minimal free resolution, and this is unique (up to isomorphism). Thus, one may speak of   the  minimal free resolution of such a module.
 We introduce now the main objects of interest.

  \vskip 2mm
 Let $M$ be a  finitely generated filtered $R$-module   and   $S=\{f_1, \dots, f_s\} $ be a system of elements  of $M $ and let  $ \nu_{\mathbb{M}}(f_i) $  be the corresponding valuations.  As in Definition \ref{L}, let $L =  \oplus_{i=1}^s R e_i $ be a free $R$-module of rank $s$ equipped with the filtration  $\mathbb{L}    $ where  $\nu_i=  \nu_{\mathbb{M}}(f_i).$   Then  we denote the filtered free $R$-module $L $ by $ \oplus_{i=1}^s R(-  \nu_{\mathbb{M}}(f_i) ), $ hence $\nu_{\mathbb{L} }(e_i)=  \nu_{\mathbb{M}}(f_i).$

 Let $\phi  : L  \to M $ be a morphism of filtered $R$-modules   defined by $$\phi  (e_i)= f_i. $$ It is clear that $\phi  $ is a morphism of filtered modules and  $gr_{\mathbb{L} }(L ) $ is isomorphic to the  graded free $G$-module $\oplus_{i=1}^s G(- \nu_{\mathbb{M}}(f_i)) $ with a basis $(e_1,\dots,e_s) $ where deg$(e_i)=  \nu_{\mathbb{M}}(f_i).$ In particular $\phi $ induces a natural graded morphism (of degree zero) $gr(\phi ) : gr_{\mathbb{L} }(L )  \to gr_{\mathbb{M} }(M) $ sending $e_i$ to $gr_{\mathbb{M}}(f_i)=f_i^*.$

 \vskip 2mm
 Let $ c= ^t (c_1, \dots, c_s) $ be an element of $L .$ By the definition of the filtration $\mathbb{L}  $ on $L, $ we have $$\nu_{\mathbb{L} }(c) = min \{ \nu_R(c_i) + \nu_{\mathbb{M }}(f_i) \ : \ 1\le i \le s \} \ \le \ \nu_{\mathbb{M}}(\phi (c)).$$
Set $gr_{\mathbb{L} }(c) = ^t (c_1', \dots, c_s')  $  and $\nu =  \nu_{\mathbb{L} }(c), $ then
 \vskip 2mm
  $c_i'=\Big\{\begin{array}{ll}
   gr_{\n}(c_i)   & \textrm{if $  \  \nu_R(c_i) + \nu_{\mathbb{M }}(f_i) = \nu $}\\
       0 & \textrm{if $    \  \nu_R(c_i) + \nu_{\mathbb{M }}(f_i) > \nu $}\end{array}$\\


 \vskip 2mm

 If we denote by $Syz (S) $ the submodule of $L$ generated by the first syzygies of $f_1, \dots, f_s, $ then
 $\Syz(S) =$Ker$ \phi. $ Likewise let  $\Syz (gr_{\mathbb{M}}(S))$ be the module generated by the first syzygies of $gr_{\mathbb{M}}(f_1), \dots, gr_{\mathbb{M}}(f_s), $ then $\Syz (gr_{\mathbb{M}}(S)) = Ker( gr (\phi)).$
\vskip 3mm
\noindent  Then we have  the following fundamental diagram:
 \vskip 2mm
$\begin{CD}0 \hspace{0.6 cm}@>>> \hspace{0. cm} \Syz(S)\hspace{0.5 cm}@> >> \oplus_{i=1}^s R(-\nu_{\mathbb{M}}(f_i)) @>{\phi}>>M    \\ \end{CD}$\\

$\hspace{5.7 cm} \downarrow gr_{\mathbb{L}} \hspace{3.0  cm} \downarrow gr_{\mathbb{M}} $\\

$\begin{CD}0 @>>>  \Syz(gr_{\mathbb{M}}(S) )@>{ }>> \oplus_{i=1}^s G (-\nu_{\mathbb{M}}(f_i))  @>{gr(\phi ) }>>  gr_{\mathbb{M}}(M)  \\ \end{CD}$\\
\newline
  \vskip 2mm

\begin{definition}\label{lift}  Let $M$ be a filtered module. An element $g \in M$ is a \emph{lifting} \rm of an element $h \in gr_{\mathbb{M}}(M) $
if  $$gr_{\mathbb{M}}(g)=h.$$
\end{definition}

\vskip 3mm
\begin{remark}\label{LFsyz} 
\rm If  $p\in \Syz(S)$, then we have $gr_{\mathbb{L}}(p) \in \Syz(gr_{\mathbb{M}}(S))$; in particular the map $gr_{\mathbb{L}} $ induces a map
$$gr_{\mathbb{L}}  {|}: \Syz(S) \longrightarrow  \Syz(gr_{\mathbb{M}}(S)).$$
\end{remark}

 \vskip 2mm
  Following the setting in \cite[Section 2]{Sh} and \cite{KR}, we introduce the concept of standard bases of a module.

 \begin{definition} Let $M$ be a filtered $R$-module. A subset $S=\{f_1,\dots, f_s\} $ of $M$ is called a {\it{standard basis}} of $M$ if $$gr_{\mathbb{M}}(M) = <f_1^*, \dots, f_s^*>.$$ If any proper subset of $S$ is not a standard basis, we call $S$ a {\it{minimal standard basis}}.
 \end{definition}

Let  $ f: M \to N $ be   a morphism of filtered $R$-modules and
$gr(f) :  gr_{\mathbb{M}}(M) \to  gr_{\mathbb{N}}(N) $ the induced
homomorphism. If $gr (f) $ is surjective, then $f$ is a strict
surjective homomorphism.  By using this fact and Proposition
\ref{RoV}, we can prove next theorem which gives a criteria for
standard bases.   In the case of ideals, the result had been  proved
in \cite{RoV}. We omit here the proof because it is essentially the
same as  Theorem 2.9 in \cite{Sh}.

\begin{theorem}\label{esiste}Let  $M$ be a filtered $R$-module, $f_1,\dots,f_s \in M$ and $S=\{ f_1,\dots,f_s\} $.
The following facts are equivalent:
\begin{enumerate}\item  $\{f_1,\dots,f_s\}$ is  a   standard basis of $M.$
\item $\{f_1,\dots,f_s\}$ generates $M $ and every element of $ \Syz(gr_{\mathbb{M}}(S)) $ can be lifted to an element  in $Syz(S).$
\item $\{f_1,\dots,f_s\}$ generates $M $ and $  \Syz(gr_{\mathbb{M}}(S))= gr_{\mathbb{L}}(\Syz(S)).$
 \end{enumerate}
\end{theorem}
 \vskip 2mm

 Note that a similar result is well known  for  Gr\"obner bases   (see for example Theorem 2.4.1 D1) and B2), \cite{KR}).
\vskip 2mm
 With the notation of the previous fundamental diagram, the equivalent conditions of Theorem \ref{esiste} are also equivalent to the following:
\vskip 2mm
  {\it{ 4.  $gr(\phi) $ is surjective}}
\vskip 2mm
  {\it { 5.  $\{f_1,\dots,f_s\}$ generates $M $ and $\phi $ is strict.}}
\vskip 3mm

In this setting it comes natural  the following result  presented  in  \cite{Rob} and also in \cite[ Theorem 3.1]{HRV},  which gives a comparison between an  $R$-free resolution of $M$ and a $G$-free resolution of  $ gr_{\mathbb{M}}(M). $ The result  will be a  central tool in our investigation and  we present here a proof in terms of standard bases because  this constructive approach will be fundamental  in the following.

\begin{theorem} \label{main}  Let $M$ be a filtered $R$-module and $({\bf{G.}}, d. )$ a  $G$-free graded   resolution of
$ gr_{\mathbb{M}}(M). $ Then we can build up an  $R$-free resolution  $({\bf{F.}, \delta.})$ of $M$ and a special filtration $\mathbb{F} $ on it such that $gr_{\mathbb{F}}({\bf{F.}}) = {\bf{G.}}.$
\end{theorem}

\begin{proof}    Let $$ {\bf{G.}}: \  \dots \to \oplus_{i=1}^{\beta_{l}} G(-a_{l i})  \overset{d_{l}} \to   \oplus_{i=1}^{\beta_{l-1}} G(-a_{l-1 i}) \overset{d_{l-1 }}  \to \dots  \overset{d_1} \to   \oplus_{i=1}^{\beta_{0}} G(-a_{0 i})  \overset{d_{0 }} \to gr_{\mathbb{M}}(M) \to  0$$ be a  $G$-free resolution (not  necessarily minimal) of $ gr_{\mathbb{M}}(M). $ We define  now $({\bf{F.}, \delta.})$ by inductive process  focused on Theorem \ref{esiste}.

We put $g_i=d_0(e_{0i}) \in  gr_{\mathbb{M}}(M)  $ and let $f_i \in M $ be a lifting of $g_i. $ Starting from the integers  $ a_{0i} = \nu_{\mathbb{M}}(f_i), $ define $F_0$ the $R$-free module of rank $\beta_0 $ with the special filtration    $$F_0= \oplus_{i=1}^{\beta_0}  R(-a_{0 i})   {\text{ and }}  \delta_0: F_0 \to M $$  such that $\delta_0(e_{0i})=f_i.$ Since $d_0 $ is surjective, then the $f_i$'s generate a standard base of $M$ hence,  by Theorem \ref{esiste}, $Ker(d_0) = gr_{\mathbb{F}_0}(Ker (\delta_0))  $   and $ F_0 \overset{\delta_0}\to M \to 0 $ is exact.

Suppose that we have defined filtered free modules $F_0, \dots, F_j$ with $0 \le j <l $ such that
$$ F_j \overset{\delta_j} \to F_{j-1} \dots \to F_0 \overset{\delta_0} \to M \to 0 $$ is a part of an  $R$-free resolution of $M. $ In particular,  for every $i \le j, $   $gr_{\mathbb{F}_i}(F_i)  = G_i , $ $gr_{\mathbb{F}_i}(\delta_i)  = d_i $ and moreover
\begin{equation} \label{step} Ker (d_j)=gr_{\mathbb{F}_j}(Ker ( \delta_j) ). \end{equation}
Let $Ker (d_j) = <g_{j 1}, \dots, g_{j \beta_{j+1}}>,$  then  there exist $f_{j i} \in Ker (\delta_j)  $ which are  lifting of $g_{ji}.$ So $ g_{j i}=  gr_{\mathbb{F}_j}(f_{j i})  $ and $ a_{j+1 i} = \nu_{\mathbb{F}_j} (f_{j i}).$ Define now the filtered free $R$-module
$$ F_{j+1} = \oplus_{i=1}^{\beta_{j+1}} R(-a_{j+1 i}) \text{ and } \delta_{j+1}: F_{j+1} \to F_j $$ such that $\delta_{j+1}(e_{ j+1 i} ) = f_{ ji}. $ Then $$ F_{j+1} \overset{\delta_{j+1}} \to F_j \overset{\delta_{j}} \to F_{j-1} $$ is a complex such that $gr_{\mathbb{F}_{j+1}}(\delta_{j+1}  )= d_{j+1}. $ Because of (\ref{step}), $ f_{j 1}, \dots, f_{j \beta_{j+1}}$ is a standard basis of $Ker(\delta_j) $ as a submodule of the filtered module $F_j $. So again by Theorem \ref{esiste}, we get $Ker (d_{j+1})=gr_{\mathbb{F}_{j+1}}(Ker ( \delta_{j+1}) ) $ and we can continue by inductive process.

\end{proof}

 It is worth  remarking that if we start from  a minimal free resolution of $gr_{\mathbb{M}}(M), $ then  the $R$-free resolution of $M$, given in the proof of Theorem \ref{main}, is not necessarily minimal and it is minimal if and only if the corresponding Betti numbers coincide, i.e. $\beta_i(gr_{\mathbb{M}}(M)) = \beta_i(M) $ for every $i \ge 0.$
In general the following inequalities hold:
\vskip 2mm
$\bullet\ \ $ $\beta_i(gr_{\mathbb{M}}(M))  \ge  \beta_i(M) {\text{ for every }} i \ge 0$
\vskip 2mm
Let $R$  be is a regular local ring and  denote by $pd(\ ) $ the projective  dimension of a module:
\vskip 2mm
$\bullet\ \ $ $ pd  (gr_{\mathbb{M}}(M)) \ge pd(M) $
\vskip 2mm
$\bullet\ \ $ $  depth\   (gr_{\mathbb{M}}(M)) \le  depth \ M. $

\vskip 2mm We are interested in finding classes of finitely
generated $R$-modules $M$  for which   the equalities hold.
Accordingly with \cite[Section 3.]{HRV} we give the following
definition.
\begin{definition} A filtered module $M$ is said to be of homogeneous type with respect to the given filtration $\mathbb{M} $ if  $\beta_i(gr_{\mathbb{M}}(M)) = \beta_i(M) $ for every $i \ge 0.$
\end{definition}

  When  the $R$-module $ M $ is of homogeneous type {\it {with respect to the $\n$-adic filtration}},   we simply say that {\it{$M$ is of homogeneous type}}.
  \vskip 2mm
  The $\n$-adic filtration has a particular interest and it produces the first interesting class  of modules  of homogeneous type: the Koszul modules introduced by J. Herzog and S. Iyengar in \cite{HI}. In fact, by  \cite[Proposition 1.5]{HI},  $M$ is a Koszul $R$-module if and only if $gr_{\n}(M) $ has a linear resolution as a $gr_{\n}(R)$-module which implies in particular that  $M$   is of homogeneous type.
  \vskip 2mm

\begin{remark} \label{l} {\rm{ Consider  $M=I$   an ideal of a regular local ring $(R, \n).   $ In this paper we will focus our attention on the following filtrations of    the ideal $I.  $

1. $\mathbb{M}= \{\n^p I\}  $ (the $\n$-adic filtration on $I$):  in this case  $gr_{\mathbb{M}}(I)= gr_{\n}(I)= \oplus_{p\ge 0} I\n^p/I\n^{p+1}.$
Accordingly with our setting, we say that $I$ is of homogeneous type if $\beta_i( gr_{\n}(I) ) = \beta_i(I) $ for every $i \ge 0.$

2. $\mathbb{M}= \{\n^p \cap I\}: $ in this case  $gr_{\mathbb{M}}(I)= I^* = \oplus_{p\ge 0} I \cap \n^p/I \cap \n^{p+1}=   \oplus_{p\ge 0}  I \cap \n^p + \n^{p+1}/ \n^{p+1} $ is the   ideal of the polynomial ring $P=gr_{\n}(R)$ generated by the initial forms of the elements of $I.$
Hence if $A=R/I$ and $\m=\n/I, $ the associated graded ring $gr_{\m}(A) \simeq P/I^*.$ According to our setting, we say that $I$ is of homogeneous type with respect to $\mathbb{M} $ if $\ \beta_i( gr_{\mathbb{M}}(I) ) = \beta_i(I) $ for every $i \ge 0.$ This is equivalent to say that $\beta_i(A)= \beta_i(gr_{\m}(A))  $  for every $i \ge 0,  $   that is
the  local ring $(A, \m) $ is of homogeneous type.  }}
\end{remark}

The following examples show how it is difficult  to find  modules of homogeneous type.
 \vskip
2mm
\begin{example}  {\rm (1) Let    $I=({{x^3}}-y^7,  {x^2y} -xt^3-z^6) $ be  in   $R= k[[x, y, z, t]]. $
    The ideal is a complete intersection and hence the resolution of $I$ as an $R$-module  is given by the Koszul complex. But $$I^*=(x^3, x^2y, x^2t^3,  xt^6,  x^2z^6, xy^9 - xz^6t^3, xy^8t^3,  y^7t^9) \subseteq P=k[x,y,z,t] $$  and hence $\beta_0(I^*)=8 > \beta_0(I)=2. $  Using    \cite{CoCoA}, it is possible to check that $\beta_1(I^*)=12 > \beta_1(I)=1, $ $\beta_2(I^*)=6 > \beta_2(I)=0, $ $\beta_3(I^*)=1 > \beta_3(I)=0. $

\vskip 2mm

\noindent The following example shows that $\beta_0(I) = \beta_0(I^*) $   and $ pd(I)= pd(I^*) $ do not force  $I $ to be  of homogeneous type.

\noindent (2) Consider the local ring $$A=k[[t^{19}, t^{26}, t^{34}, t^{40}]]= k[[x,y,z,t]]/I, $$ one can prove that $I$ is minimally generated by an  $\n$-standard base, i.e. $\beta_0(I)= \beta_0(I^*)=5,  $ $I$ and $I^*$ are perfect ideals (hence they have the same projective  dimension), nevertheless $I$ is not of homogeneous type with respect to $\mathbb{M}= \{\n^p \cap I\},$  neither  $A$ is of homogeneous type (see \cite[Example (3)]{HRV}). }

\end{example}
\vskip 2mm

\noindent Nevertheless examples of local rings  of homogeneous type (not necessarily Koszul) can be  given.

\begin{example} \label{examples} {\rm{ (1) Let $I$ be an ideal of $R$ generated by a super-regular sequence. This means that $I=(f_1,\dots, f_r) $ where $f_1, \dots, f_r $ is a regular sequence and an  $\n$-standard base of $I$, equivalently the initial forms $f_1^*, \dots, f_r^* $ are a regular sequence in $P=gr_{\n}(R) $ (see \cite{VV}). Then both $A=R/I $ and $I$ are of homogeneous type (see \cite[Example 1, Theorem 3.6]{HRV}).

(2) Let $I$ be the ideal generated by the maximal minors of a generic $ r \times s $ ($r \le s$) matrix $X= (x_{i j}) $ in $R=k[[ x_{ij}]], $ then $gr_{\n}(I) \simeq I (-r) $ has a linear resolution and it is easy to prove that $I$ is of homogeneous type.

(3) Let $I$ be an ideal of the regular ring $(R, \n) $ such that $
A= R/I    $ is Cohen-Macaulay of minimal multiplicity and let  $ \m=
\n/ I $, then J. Sally (see \cite{Sa}) proved that $gr_{\m}(A)
$    is Cohen-Macaulay of minimal degree  and $I $ has a standard
base of equimultiple elements of degree $2.$  From this, using
\cite[Lemma 3.3]{HRV}, one can  prove that $I$ is of homogeneous
type.

(4) Let $I$ be an ideal of the regular ring $(R, \n) $  generated by two elements, then $I$ is of homogeneous type (see \cite[Proposition 3.4]{HRV}).

(5)  Let $I$ be the defining  ideal of a monomial curve in $\mathcal{A}^3$ in  the regular ring $R$  such that $\nu(I)=\nu(I^*),$  then $A=R/I $ is of homogeneous type.  Because $I$ is a perfect ideal of  codimension two, it is enough to recall that  Robbiano and Valla in \cite{RoV1} proved that, in this case,  $gr_{\m}(A) $ is Cohen-Macaulay.   }}

\end{example}

We remark that Proposition 2.4. in \cite{RS}   gives us a criterion for  producing   more modules of homogeneous type.

\section{Properties of componentwise linear modules }

The componentwise linear modules over a polynomial ring had been introduced by Herzog and Hibi by enlarging the class of the graded modules with a $d$-linear resolution. Interesting results concerning their graded Betti numbers had been proved by Aramova, Conca, Herzog and Hibi (see \cite{H,HH,AHH,CHH,C}). Later R\"omer (see \cite{R}) studied more homological properties of the componentwise linear graded modules in the  general setting of finitely generated modules over  Koszul algebras (instead of polynomial rings),    some of them  partially overlap with those of Martinez and Zacharia in \cite{MZ}. Thanks  to the fact that componentwise linear modules and graded Koszul modules coincide, in the literature one can find two different approaches: the first coming from Herzog and Hibi's methods (see \cite{HH}, \cite{CHH}) dealing with graded ideals in the polynomial ring and  a  purely homological approach (see \cite{MZ} and \cite{HI}) on modules over Koszul algebras.  However, in view of the applications,  in this section we consider graded modules over a polynomial ring $P=k[x_1, \dots, x_n] $ even if most of the results hold in a more general setting.
 \vskip 2mm
 Let $N $  be a graded $P$-module.  For $d \in {\bf Z} $  we write $ N_{<d>} $ for
the submodule of $N$  which is generated by all homogeneous elements of $ N $ with
degree $ d.$  In the graded case we may also define  the graded Betti numbers, i.e.
 $$\beta_{i, j}(N) := dim_k Tor^P_i(k, N)_j.$$  If it will be clear the context, we will simply write  $\beta_{i,j}.$

 \begin{definition}  Let  $N$   be a graded $P$-module.

(i) Let $d \in{\bf{ Z}}.$  Then $ N$ has a $d$-linear resolution if  $\beta_{i, j} = 0 $ for $j \neq d+i.$

(ii) $ N $  is componentwise linear if for all integers $d $ the module $ N_{<d>} $ has a $d$-linear
resolution.
 \end{definition}

 For   more information concerning the componentwise linear modules,  see \cite{HH, C, R, CHH}. We select here some good properties of their  graded minimal free resolutions.

\vskip 2mm
 Set  $ indeg  (N) = min \{ d \in{ \bf{Z}} \ : \ N_{d} \neq 0. \} $  If $N$ is componentwise linear,  it is known that $ N/N_{<indeg(N)>} $ is componentwise linear too (see \cite[Lemma 3.2.2.]{R}). Let  $ ({\bf{G. }}, d.) $ be  the minimal graded free resolution of $ N  $ and define the subcomplex  $ ({\widetilde{\bf{G. }}, \widetilde{ d.}}) $  of $ ({\bf{G. }}, d.) $   by  $$ \widetilde{G_i} = P ( -(i+indeg(N)))^{\beta_{i, i +indeg(N)}} \subseteq G_i  \text{ and }  \widetilde{ d.}= d.|_{\widetilde{\bf{G. }}}.$$  R\"omer proved that  $ \ {\widetilde{\bf{G. }}}$ is the{\it{ minimal}}  (linear) graded free resolution of $N_{<indeg(N)>} $ and ${\bf{G. }}/\widetilde{\bf{G. }} $ is the minimal graded free resolution of $N/ N_{<indeg(N)>}  $ (see  \cite[Lemma 3.2.4.]{R}).

 As a consequence of these properties we  easily get the following  information that  have an intrinsic interest in the theory of componentwise linear modules.

 \begin{proposition} \label{resolution}  Let  $ N $ be a graded $ P$-module minimally generated in degrees  $ i_1, \dots, i_m. $     Assume $ N $ is componentwise linear  and  let  $ ({\bf{G.}}, d.) $ be    the minimal graded free resolution of $ N. $ Then for every $1 \le s \le pd(N) $ we have $$ Tor_s^P(k, N)_j=0  \text{  for }  j \neq i_1+s, \dots, i_m +s. $$
   In particular,  if for some $\ \ \ 1 \le s \le pd(N) \  \ $  and $\ \ 1 \le r \le m,\ \ \  $   $   Tor_s^P(k, N)_{i_r+s}=0,  \  $     then  $\   Tor_{s+1}^P(k, N)_{i_r+s+1}=0. $
 \end{proposition}
 \begin{proof}
  Let $N_1=N_{<indeg(N)>}$.
   Since $N$ is componentwise linear, $N_1$ has $i_1=indeg(N)$-linear resolution and $\overline{N}=N/N_1$ is a componentwise linear module minimally generated in degrees  $ i_2, \dots, i_m$. Thus, from the short exact sequence
   $$0\to N_1\to N\to \overline{N}\to 0$$ we obtain $Tor_i^P(k,N)_j=Tor_i^P(k,N_1)_j\oplus Tor_i^P(k,\overline{N})_j$.   Repeating this procedure,
   we can find the sequence of graded modules $N_1,..., N_m$, such that $N_i$ has an $i_r$-linear resolution and
   $$Tor_i^P(k,N)_j=\oplus_{r=1}^mTor_i^P(k,N_r)_j.$$
   Hence, the conclusion follows.
  \end{proof}

  The following remark will clarify the shape of the matrices associated to the differential maps of the resolutions of componentwise linear modules.

  \begin{remark}{\label{matrici}} {\rm{ Let  $N $   be a graded $P$-module generated in degrees  $i_1, \dots, i_m. $ Assume $ N$  is componentwise linear  and  let  $ ({\bf{G.}}, d.) $ be    the minimal graded free resolution of $ N. $ Then by the above proposition  $G_s=  \oplus_{j=1}^m P^{\beta_{s, i_j +s}} (- (i_j +s))  $ for every $1 \le s \le pd(N).$ We want to describe the shape of the matrix $\mathcal{M}_s $ associated to
  $ G_s \overset{d_s} \to G_{s-1} $ with respect  to the canonical homogeneous bases of $G_s $ and $G_{s-1} $ of degrees  respectively $i_1+s, \dots, i_m+s $ and $i_1+s-1, \dots, i_m+s-1. $

By using Proposition \ref{resolution},  without
loss of generality, we may assume $\mathcal{M}_s $ of the following
shape:

   \xymatrix@C=   2ex@R=0.5ex{
 & &                i_1+s       & &        i_2+s          & &                    & &             i_m+s   & && \\
 & & \hspace{7ex}\drop\frm{^\}} & & \hspace{7ex}\drop\frm{^\}} & & & & \hspace{7ex}\drop\frm{^\}} &  &  \\
 & \ar@{-}@/_1pc/[dddddddd] & & \ar@{.}[dddddddd]  &   &  \ar@{.}[dddddddd]  & &   \ar@{.}[dddddddd]  &  &  \ar@{-}@/^1pc/[dddddddd]   &  \\
i_1+s-1 \,\bigg\{ & &     B_{i_1 i_1 s}           & &    B_{i_1 i_2 s}    & &       \ldots         & &             B_{i_1 i_m s} &  \\
 & \ar@{.}[rrrrrrrr] &&&&&&&&& \\
i_2+s-1 \,\bigg\{ & &      0               & &    B_{i_2 i_2 s}    & &       \dots       & &          B_{i_2 i_m  s }& & \\
  & \ar@{.}[rrrrrrrr] &&&&&&&&& \\
                      & &        0               &&        0          &&            \ddots                &&          \vdots              && \\
  & \ar@{.}[rrrrrrrr] &&&&&&&&& \\
 i_m +s-1 \,\bigg\{ & &               0        & &       \dots         & &               0            & &      B_{i_m  i_m  s }  & & \\
 &&&&&&&&&&
}

\noindent where all  the non zero entries of $  B_{i_1 i_1 s}, B_{i_2 i_2 s}, B_{i_m  i_m  s } $ (diagonal blocks) are linear forms and the non zero entries of $ B_{i_p  i_q  s } $ with $ p < q$  (up-diagonal blocks) are forms of degree at least  two.

}}

\end{remark}

\begin{remark} {\label{property}}  {\rm{Let $ N$   be a componentwise linear graded $P$-module. Set $ N_1: =N  $ and for every $j=2, \dots, m $ we define $$N_j  := N_{j-1} /(N_{ j-1})_{<indeg(N_{j-1}>}  =N_{j-1}/ (N_{ j-1})_{<i_{j -1}>}.  $$

\noindent By  \cite[Lemma 3.2.2  and Lemma 3.2.4.]{R}, it is easy to show that, for every  $1 \le s \le pd(N),  $ the  matrices $ B_{i_j i_j s},  $   (on the diagonal) in $\mathcal{M}_s$ have the following properties :
\vskip 2mm
$\bullet$   in each column of  $ B_{i_j i_j s}  $ at least one entry is different from zero.
\vskip 2mm
$\bullet$ the columns of $ B_{i_j i_j s}   $ minimally generate the $s$-th syzygy module of $(N_j)_{<i_j>}.$
\vskip 2mm
Indeed it is enough to remark that the matrices $ B_{i_j i_j s},  $ can be considered as the matrices associated to the differential maps of  the minimal free resolution of $(N_j)_{<i_j>} $ which has linear resolution. }}
\end{remark}
\vskip 3mm
We present the following example in order to help the reader to visualize better the resolutions of componentwise linear modules.

\begin{example} {\rm{ Let $P=k[x_1, x_2,x_3,x_4] $ and $I=(x_1^2, x_1x_2, x_2^2, x_1x_3, x_2x_3^2, x_1x_4^3, x_3^4).$ The ideal $I$ is Borel-fixed, so $I$ is componentwise linear and the minimal free resolution of $I$ is:

\noindent $0 \to P(-7) \overset{d_3} \to P(-4)\oplus P(-5) \oplus P^4(-6) \overset{d_2}\to P^4(-3)\oplus P^2(-4) \oplus P^5(-5) \overset{d_1} \to P^4(-2)\oplus P(-3) \oplus P^2(-4) \to I \to 0$

\noindent According to    Remark \ref{matrici}, we have }}
 \[
 \xymatrix@C=   0.3ex@R=0.1 ex{
 &\ar@{-}@/_0.8pc/[dddddddddd] & & & & &  \ar@{-}[dddddddddd] & & & \ar@{-}[dddddddddd] & & & & & & \ar@{-}@/^0.8pc/[dddddddddd]\\
 && x_2 & x_1 & 0 & 0 &     & 0 & 0 &        & x_4^3 & 0 & 0 & 0 & x_3^3 & \\
 && 0 & 0 & x_1 & 0 &        & x_3^2 & 0 &      & 0 & 0 & 0 & 0 & 0 & \\
 &&-x_3 & 0 & -x_2 & x_1 &    & 0 & x_3^2 &    & 0 & x_4^3 & 0& 0& 0 & \\
{\mathcal{M}_1} = \ \ && 0 & -x_3 & 0 & -x_2 &   & 0 & 0 &      & 0 & 0 & x_4^3 & 0 & 0 & \\
&  \ar@{-}[rrrrrrrrrrrrrr] &&&&&&&&&&&&&& \\
 &&0&0&0&0&    &-x_2&-x_1&    &0&0&0&x_3^2&0& \\
 &  \ar@{-}[rrrrrrrrrrrrrr] &&&&&&&&&&&&&& \\
   &&0&0&0&0&    &0&0&    &0&0&0&-x_2&-x_1& \\
   &&0&0&0&0&    &0&0&    &-x_3&-x_2&-x_1&0&0& \\
 &&&&&&&&&&&&&&&
 }
\]
   \vskip 2mm

\noindent {\rm One can find a similar shape for  ${\mathcal{M}_2} $ and ${\mathcal{M}_3}.$ }

\end{example}

\vskip 1cm

\section{Extremal Betti numbers}

In this section, we present  the main result of the paper and the application to the minimal free resolutions  of a local ring. We denote by $\mu(\ ) $ the minimal number of generators of a module over  a local ring (or the minimal number of generators of a graded module over the polynomial ring).

\begin{theorem} \label{M} Let $M$ be a finitely generated filtered module over a regular local ring $(R, \n).$ Assume that
\begin{enumerate}
\item $\mu(M) = \mu(gr_{\mathbb{M}} (M)) $
\item $ gr_{\mathbb{M}} (M) $ is a componentwise linear  $P$-module.
\end{enumerate}
Then $M$ is of homogeneous type with respect to $\mathbb{M}.$
\end{theorem}

\begin{proof} For short we denote  $ gr_{\mathbb{M}} (M) $ by $  M^* . $    By Corollary \ref{resolution},  its    minimal graded free resolution     $({\bf {G.}},  d.  )$   as  a $P=gr_{\n}(R)$-module has the following shape:
$$ 0 \to \oplus_{j=1}^m P^{\beta_{h, i_j+h}} (-(i_j +h)) \overset{d_h}\to \dots \overset{d_1} \to \oplus_{j=1}^m P^{\beta_{0, i_j }} (- i_j) \to M^* \to 0 $$
where $h=pd(M^*) $ and $0 < i_1<i_2< \dots < i_m $ are  the degrees  of a minimal system of generators of $M^*. $ Denote by $\beta_t (M^*) $ the total Betti numbers of $M^*.$  From  $({\bf {G.}},  d.)$  we can build up a free resolution $({\bf{F.}, \delta.})$ of $M$ by the inductive process described in Theorem \ref{main}:
$$ 0 \to \oplus_{j=1}^m R^{\beta_{h, i_j+h}} (-(i_j +h)) \overset{\delta_h}\to \dots \overset{\delta_1} \to \oplus_{j=1}^m R^{\beta_{0, i_j }} (- i_j) \to M \to 0. $$
We have to prove  that $({\bf{F.}, \delta.})$  is  minimal.   For every $t=0, \dots, h $ denote by ${\mathcal{M}}^*_{t} $ (resp. ${\mathcal{M}}_{t} $) the matrix of the differential map $d_{t} $ (resp. $\delta_{t} $).  We   prove  that the columns of $ {\mathcal{M}}_t $ for $t=0, \dots, h $ {\it{ minimally }} generate the $t$-th syzygy module of $M,  $  that is $Ker(\delta_{t-1}).$   We  proceed  by   inductive process.
The first step ($t=0$) follows easily from the    assumption $\mu(M) = \mu(M^*), $  which says that a minimal system of generators of $M^*$ (say $g_{-1,i}$) build up a minimal system of generators of $M $ (say $f_{-1,i}$). Suppose now that for each $ 0 \le t \le s < h$ we have proved that
$$F_s=\oplus_{j=1}^m R^{\beta_{s, i_j+s}} (-(i_j +s)) \overset{\delta_s}\to \dots \overset{\delta_1} \to F_0=\oplus_{j=1}^m R^{\beta_{0, i_j }} (- i_j) \overset{\delta_0}\to M  \to 0 $$ is part of a {\it{minimal}} free resolution of $M. $  This means in particular that,  if $d_s(e_{s,r}))=g_{s-1,r}$ and $f_{s-1, r} $ is the corresponding lifting in the filtered module $F_{s-1} $ ($F_{-1}=M$), then $\{f_{s-1,1}, \dots, f_{s-1, \beta_s} \} $ is a minimal  generating set for the $s$-th syzygy module. By following  the proof of Theorem \ref{main},
we have to prove now that
$$ F_{s+1} = \oplus_{i=1}^{\beta_{m}} R^{\beta_{s+1, i_j+s+1}} (-(i_j +s+1)) \overset{ \delta_{s+1}} \to  F_{s } \to F_{s-1}   $$  is part of a minimal free resolution of $M.$ Because $\{f_{s-1,1}, \dots, f_{s-1, \beta_s} \} $ is a minimal  generating set for the module their generated, we conclude that all entries of ${\mathcal{M}}_{s+1}   $ belong to $\n.$ The goal is to prove that  the columns of the matrix  ${\mathcal{M}}_{s+1}   $ {\it{ minimally }} generate  the $s+1$-th syzygy module of $M.$

Since $M^*$ is componentwise linear, accordingly with Remark \ref{matrici},  we may assume that $\mathcal{M}_{s+1}^*  $  has  the following shape:

    \xymatrix@C=   1.0ex@R=0.15ex{
 & &                i_1+s+1        & &        i_2+s+1           & &                    & &             i_{m}+s+1   & && \\
 & & \hspace{7ex}\drop\frm{^\}} & & \hspace{7ex}\drop\frm{^\}} & & & & \hspace{7ex}\drop\frm{^\}} &  &  \\
 & \ar@{-}@/_1pc/[dddddddd] & & \ar@{.}[dddddddd]  &   &  \ar@{.}[dddddddd]  & &   \ar@{.}[dddddddd]  &  &  \ar@{-}@/^1pc/[dddddddd]   &  \\
i_1+s  \,\bigg\{ & &     M_{i_1 i_1 s+1}^*           & &    M_{i_1 i_2 s+1}^*    & &       \ldots         & &             M_{i_1 i_m s+1}^* &  \\
 & \ar@{.}[rrrrrrrr] &&&&&&&&& \\
i_2+s  \,\bigg\{ & &      0               & &    M_{i_2 i_2 s+1}^*    & &       \dots       & &          M_{i_2 i_m  s+1}^* & & \\
  & \ar@{.}[rrrrrrrr] &&&&&&&&& \\
                      & &        0               &&        0          &&            \ddots                &&          \vdots              && \\
  & \ar@{.}[rrrrrrrr] &&&&&&&&& \\
 i_m +s  \,\bigg\{ & &               0        & &       \dots         & &               0            & &      M_{i_m  i_m  s+1}^*  & & \\
 &&&&&&&&&&
}

  The columns of $\mathcal{M}_{s+1}^*  $ are the initial forms of the columns of $\mathcal{M}_{s+1} $ with respect to the special filtration $(F_s)_p = \oplus_{j=1}^m \oplus_{k=1}^{\beta_{s, i_j+s}} \n^{p-i_j-s} e^s_{k, i_j} $ on $F_s.$  In particular by looking the degree matrix of $\mathcal{M}_{s+1}^*,  $ the elements of  $ M_{i_p i_q  s+1}^*$ have degree $i_q-i_p +1 $ and   the matrices $ M_{i_j i_j s+1}^*  $ (on the diagonal) have the good properties described in   Remark \ref{property}.

Denote now by $ M_{i_p i_q s+1}$ the   blocks corresponding to the rows labeled by $i_p +s $ and the columns labeled by $ i_q+s+1  $ in $\mathcal{M}_{s+1}. $ Remark that in  $\mathcal{M}_{s+1}  $ the non zero entries of the  blocks  $ M_{i_p i_q s+1}$    have  valuation at least   $i_q-i_p +1.$ Hence if $ p <q $ the non  zero elements have valuation at  least $2. $   Notice that the non zero entries of the blocks with $p >q $ (under the diagonal blocks) have valuation $\ge 1 $ because all the entries belong to $ \n. $

Denote by $C_{i_k j} $ with $k=1, \dots, m $ and $ j=1, \dots, \beta_{s+1, i_k+s+1} $ the columns of  $\mathcal{M}_{s+1}  $ (respectively $C_{i_k j}^*  $ those of  $\mathcal{M}_{s+1}^* $).  By Remark \ref{property} in each column $C_{i_k j} $ there is at least one element of valuation $1$ and it belongs to the block $M_{i_k i_k s+1}.$ Suppose now that there exists $ (k, j) $ such that $$C_{i_k j} =\sum_{r=1}^m \ \ \sum_{p=1 }^{\beta_{s+1, i_r+s+1}}  \lambda_{(i_r, p)}
C_{i_r  p} $$  with $(i_r, p) \neq (i_k, j) .$ Because in $C_{i_k j} $ there is at least one element of valuation $1  $ and the entries of $M_{i_k i_r s+1} $ with $r >k $  have valuation at least $2, $ necessarily there exists an integer $u$ with $1\le u \le k$   such that $  \lambda_{(i_u, p)} \not \in \n $ for some $p. $ Assume $u$ the least integer with such property.  This  leads to prove  that the columns of $M_{i_u i_u s+1}^* $ are not linearly independent against   Remark \ref{property}.  Assume $u=k  $ and let $ C_{i_k p_1}, \dots,  C_{i_k p_t} $ be columns corresponding to invertible coefficients $\lambda_{(i_k, p_1)}, \dots, \lambda_{(i_k, p_t)} $  (in the band $i_k +s +1$).
    By using again that  the entries of $M_{i_k i_r s+1} $ with $r >k $  have valuation at least $2, $ one can easily prove that
 the columns   of $M_{i_k i_k  s+1}^* $ corresponding to $ (i_k, j), (i_k,  p_1), \dots, (i_k,  p_t) $ are not linearly independent. In the same way if $u<k $ we repeat the same argument    on $M_{i_u i_u  s+1}^*  $ and we get the conclusion.

\end{proof}

 \vskip 2mm

One of the main goal of the paper is the application  of the above result  to an ideal  $I$  of the regular local ring $(R, \n).$ In particular, we are interested in comparing the numerical invariants of the minimal $R$-free resolution of  $A=R/I$ and those of the minimal graded $P$-free resolution of $gr_{\m}(A) = P/ I^*$ where $\m= \n/I$ and $I^*$ is the graded  ideal generated by the initial forms of $I.$ We recall that if we apply  the general theory on filtered modules  to $M=I$ and $\mathbb{M}= \{\n^p \cap I\}, $ as we described in Remark \ref{l}, we obtain $gr_{\mathbb{M}}(M)= I^*.$

\vskip 2mm
As we have already said  \begin{equation} \label{<} \beta_i(R/I ) \le \beta_i( P/I^*). \end{equation}  Theorem \ref{M} says that if $I$ is minimally generated by an $\n$-standard base and $I^*$ is componentwise linear, then the equality holds.  Notice that in Theorem \ref{M} the  assumption which $I$ is minimally generated by an $\n$-standard base is necessary and  it is not a consequence of the assumption  that $I^*$ is componentwise linear. For example, if $I$ is the defining ideal of $A=k[[ t^{10}, t^{19}, t^{21}, t^{53}]],  $ then one can check that $I^* $  is componentwise linear, but $ \mu(I)=5 $ and $ \mu(I^*)=7. $

\vskip 2mm
The   inequality (\ref{<})  and Theorem \ref{M}  suggest  us new
 upper bounds coming from the homogeneous context. From now on assume the residue field $k$ of characteristic $0.$ We have two monomial ideals canonically attached to $I^*: $ the generic initial ideal  with respect to   the   revlex order and the lex-segment ideal of $I^*$ characterized by Macaulay's theorem.  They play a fundamental role in the investigation of many algebraic, homological, combinatorial and geometric properties of the ideal  $I$  itself. We denote  $$ Gin(I):= Gin (I^*) \ \text{ and } \ Lex(I) := Lex(I^*). $$ For the first equality, notice that in   \cite{B}    it is proved that if $R=k[[x_1,\dots,x_n]], $ one can define  an anti-degree-compatible ordering   on the terms of $R$ such that the leading term  ideal of $I, $ after performing a `generic change'  of coordinates, is a monomial ideal  which coincides with $Gin(I^*). $   The second equality  is clear from Macaulay's Theorem because the Hilbert function of the local ring $A=R/I$ is the Hilbert function of $ gr_{\m}(A)=P/I^*.$  All the involved   monomial ideals have  the same Hilbert function,    indeed  $$ HF_A(n) = HF_{gr_{\m}(A)}(n)= HF_{P/I^*}(n) = HF_{P/Lex(I)}(n)=  HF_{P/Gin(I)}(n), $$ nevertheless,  since $ \beta_i(R/I^* ) \le      \beta_i( P/Gin(I^*))   \le \beta_i( P/Lex(I^*)),   $  we have
\begin{equation} \label{inequality}  \beta_i(R/I ) \le      \beta_i( P/Gin(I))   \le \beta_i( P/Lex(I)) \end{equation}
for every $i \ge 0. $    The first inequality follows  by standard deformation argument, the second was proved by A. Bigatti \cite{B} and H.A. Hulett \cite{Hu} in characteristic zero and extended later by K. Pardue \cite{P} to positive characteristic.

\noindent   Componentwise linear ideals have been characterized by A. Aramova, J. Herzog and T. Hibi in \cite{AHH} as those ideals having the same Betti numbers as their generic initial ideal.

\begin{theorem} \label{HH} \rm{\cite[Theorem 1.1.]{AHH}}.
Let $J$ be a  homogeneous ideal of $P.$ The following facts are equivalent:

i) $\mu(J)= \mu(Gin(J)) $

ii) $\beta_i(J)= \beta_i(Gin(J)) $ for every $i \ge 0.$

iii) $\beta_{ij} (J)= \beta_{i j} (Gin(J))  $ for every $i, j \ge 0.$

iv) $J$ is componentwise linear
\end{theorem}

\noindent Generalization of this result have been proved in  \cite{CHH, C, P}.

\vskip 2mm Let now $I $ be an ideal in the local ring $R  $ and we  present   the similar result in the local setting.
Since $\mu(I) \le  \mu(I^*)  \le \mu(Gin(I)),   $  as a corollary of Theorem \ref{M} and the above result, we deduce  the following characterization.

\begin{corollary} {\label{I*}}  Let $I$ be an  ideal of the regular local ring $(R, \n).$ The following facts are equivalent:

i) $\mu(I)= \mu(Gin(I)) $

ii) $\beta_i(I )= \beta_i(I^*) = \beta_i(Gin(I))$ for every $i \ge 0.$

iii) $\mu(I)= \mu(I^*) $ and  $ I^* $ is componentwise linear
\end{corollary}

J.  Herzog  and T. Hibi ( \cite[Corollary 1.4.]{HH}) proved the corresponding result of Theorem \ref{HH}  for the lex-segment ideal associated to a homogeneous ideal of $P.$ The   result  is  very interesting  because in general it is easier to determine $Lex(I) $ (it is uniquely  determined by the Hilbert function) than  $Gin(I).$

\noindent Similarly, starting from the inequalities   $\mu(I) \le  \mu(I^*)  \le \mu(Gin(I)) \le  \mu(Lex(I)),  $  it is easy to deduce:

\begin{corollary} \label{LEX} Let $I$ be an  ideal of the regular local ring $(R, \n).$ The following facts are equivalent:

i) $\mu(I)= \mu(Lex(I)) $

ii) $\beta_i(I )= \beta_i(I^*) = \beta_i(Lex(I))$ for every $i \ge 0.$

iii) $\mu(I)= \mu(I^*) $ and  $ I^* $ is a Gotzmann ideal.
\end{corollary}

 \begin{example}  {\rm{ (1).  \ Let $(A, \m, k)$ be a stretched Cohen-Macaulay ring of embedding codimension $h.$   Sally in  \cite{Sa}  defined the stretched Cohen-Macaulay ring as the local rings which admit an Artinian reduction $ B$ with Hilbert function $H_B(i) \le 1 $ for $i \ge 2.$ For example this is the case if $A$ is Cohen-Macaulay of  multiplicity $\le h+2. $

If $A$ has maximal Cohen-Macaulay  type (i.e. $h$) and $gr_{\m}(A) $
is Cohen-Macaulay, then $A$ is of homogeneous type. We may assume
$A=R/I $ where $R$ is a regular local ring. By reducing the problem
to the Artinian reduction $B$, it is known (see \cite{Sa} and
\cite{EV}) that $\mu(I)  =  {{h+1}\choose 2}. $ In particular
$\mu(I)= \mu(Gin(I)), $ then   $B$ (hence $A$)  is of homogeneous
type by Corollary \ref{I*}. \vskip 2mm \noindent (2).   The local
ring $A= k[[t^9, t^{17}, t^{19}, t^{39}]]   $ is of homogeneous
type. In this case the defining ideal $$I=(x_2x_3-x_1^4,
x_2^5-x_1x_3^4, x_2x_4-x_1^2x_3^2, x_3^2x_4-x_1x_2^4,
x_3^3-x_1^2x_4, x_4^2-x_1^3x_2^3)$$ is minimally generated by a
standard base. Moreover $Gin(I)= (x_1^2, x_1x_2, x_2^2, x_1x_3^2,
x_2x_3^2, x_3^5), $ hence $\mu(I)=\mu(I^*)= \mu(Gin(I) =6 $ and we
may apply Corollary \ref{I*}. Notice that $A$ is not stretched
because $H_B(2)=3.$

}}
 \end{example}
\vskip 3mm
 A. Conca, J. Herzog and T. Hibi  proved that the Betti numbers of  an ideal $I$ of a regular    local ring $R  $ of dimension $n$   can be related   to another sequence of numbers, $\alpha_1(I), \alpha_2(I), \dots $ called the  {\it generic annihilator numbers} of $A=R/I. $ Assuming that the residue class field is infinite, regular system of parameters $y_1, \dots, y_n$ can be chosen such that for every $p=1, \dots n $ $$A_p := (y_1, \dots, y_{p-1})A :_A y_p/ (y_1, \dots, y_{p-1})A$$ is of finite length. Denoting by $$\alpha_p(A) := length A_p,  $$ in \cite[Corollary 1.2]{CHH} it was proved that $$\beta_i(A) \le \sum_{j=1}^{n-i+1} {{n-j}\choose {i-1} } \alpha_j(A).$$

If $A$ is a graded standard algebra, in \cite{CHH} it is proved that the equality holds  provided the homogeneous ideal $I$ is componentwise linear. In the local case we loose this characterization, but by  \cite[Theorem 1.5 and Remark 1.6]{CHH}, the equality holds if $I$ is Koszul or equivalently $gr_{\n}(I) $ has a linear resolution.  By  \cite[ Proposition 1.5]{HI}, for proving that a module is Koszul, it will be useful to introduce the {\it{ linearity  defect} } denoted by $ld.$


 \vskip 3mm
 As usual let  (${\bf{F.},  }\delta. $)  be a minimal $R$-free resolution of a module $M.$  For all integer $i$ we have
 $$gr_{\n}(F_i)(-i) = \oplus_{j\ge i} \n^{j-i} F_i/  \n^{j+1-i} F_i \simeq gr_{\n}(R)^{\beta_i(M)}(-i) $$
Following this  construction due to  D. Eisenbud,  G. Floystad and F. O. Schreyer in \cite{EFS}, the differential maps $\delta_i $ induces a bihomogeneous map :
 $$ \delta^{lin}_{i+1} : gr_{\n}(F_{i+1})(-i-1)  \to gr_{\n}(F_i)(-i) $$
 which can be described by matrices of {\it{linear forms}}. Precisely the matrices, say ${\mathcal{M}}_{i+1}^{lin}, $  are obtained by replacing in   ${\mathcal{M}}_{i+1} $  all entries of valuation $> 1$ by $0 $ and by   replacing all the  entries of valuation one by their  initial forms with respect to the $\n$-adic filtration.  The minimality of  (${\bf{F.} }, \delta.$)  ensures that the maps $\{ \delta^{lin}_{i} \}$  are well-defined  and  form a complex homomorphism  denoted by $lin^R({\bf{F.}}) $ which is not necessarily exact. It is called the {\it {linear part of the resolution. }}   For the construction of this complex and related results see  \cite{EFS},  as well \cite{HI, R}.
   T. R\"oemer introduced a measure for the lack of the exactness and he defined
  $$ ld( M) := inf \{j : H_i(lin^R({\bf{F.}})) =0  \text{ for }  i \ge j+1 \} $$
 In particular $lpd(M)=0 $ if and only if    $lin^R({\bf{F.}}) $ is exact. Following \cite{HI},  $M$ is Koszul if and only if $ld(M)=0. $ R\"oemer proved in \cite[Theorem 3.2.8]{R} that,  for graded modules,  having $ld(M)=0 $ is equivalent to be    componentwise linear and hence to be Koszul.
 Herzog and Iyengar proved in \cite[Proposition 1.5.]{HI}  that  to be {\it {Koszul}}  is equivalent to the fact that $lin^R({\bf{F.}}) $ is the minimal free resolution of $gr_{\n}(M)= \oplus_j \n^j M/\n^{j+1} M.   $ In particular  this is the case if and only if  $gr_{\n}(M)  $ has  a linear resolution as a  $gr_{\n}(R)$-module.

 \vskip 2mm
 The following theorem says that under the assumptions  of Theorem \ref{M}, the module $M$ is Koszul, hence $M$ is of homogeneous type (w.r.t. the $\n$-adic filtration).

 \begin{theorem} \label{lpd} Let $M$ be a finitely generated filtered module over a  regular local ring $(R, \n) $ such  that  $\mu(M) = \mu(gr_{\mathbb{M}} (M)).  $ 
 
\noindent If $ gr_{\mathbb{M}} (M) $ is a componentwise linear  $P$-module (equivalently a Koszul graded module), 
then $M$ is Koszul. 
 \end{theorem}
 \begin{proof}  We prove $ld(M)=0$.
    Let $M^*=gr_\mathbb{M}(M)$ and   (${\bf{G.},  } d. $)  be a minimal $P$-free resolution of $M^* $ (where $P=gr_{\n}(R)$).
    Denote by $0 <i_1< \dots < i_m $ the degrees of the elements of  a minimal  set of generators of $M^*.$

 \noindent From ${\bf{G.}} $  we can build   $lin^P({\bf{G.}}) $ as defined before.  Since $M^*$ is a componentwise linear module, $lin^P({\bf{G.}}) $
 is exact ($ld(M^*)=0$). Moreover, we can split  $ lin^R({\bf{G.}}) $  as $\oplus_{r=1}^m   lin^P({\bf{G_{i_r}.}}) $ where $ {lin^P(\bf{G_{i_r}.}})  $ is the linear part of the resolution of   the submodule of $M^*$ generated by  the minimal set of generators of $M^*$ of degree $i_r.$

 In fact,  by the construction, the matrices  $ lin (\mathcal{M}_{j}^*)  $  associated  to $$ d_j^{ lin}: P^{\beta_j(M^*)} (-j ) \to P^{\beta_{j-1}(M^*)} (-j+1) $$     are obtained from  $ \mathcal{M}_{j}^*  $ by replacing all the entries of degree $>1 $ by $0.$ Then,
 by Remark \ref{matrici},  the matrices  $ lin (\mathcal{M}_{j}^*),   $     $1 \le j \le pd(M) $ will present  the following shape:

 \[
    \xymatrix@C=   1.5ex@R=0.25ex{
 & &                i_1+j        & &                    & &             i_{m}+j   & && \\
 & & \hspace{7ex}\drop\frm{^\}} &  & & & \hspace{7ex}\drop\frm{^\}} &  &  \\
 & \ar@{-}@/_1pc/[dddddd] & & \ar@{.}[dddddd]  &    &   \ar@{.}[dddddd]  &  &  \ar@{-}@/^1pc/[dddddd]   &  \\
i_1+j-1 \,\bigg\{ & &     M_{i_1 i_1 j}^*           & &     0       & &            0 &  \\
 & \ar@{.}[rrrrrr] &&&&&&& \\
                      & &        0               &&            \ddots                &&         0            && \\
  & \ar@{.}[rrrrrr] &&&&&&& \\
 i_m +j-1  \,\bigg\{ & &               0        & &               0            & &      M_{i_m  i_m  j}^*  & & \\
 &&&&&&&&
}
\]

We remark that all the   blocks on the diagonal  are  the same as in  $ \mathcal{M}_{j}^*  $ (they have linear entries) and the upper diagonal blocks are replaced by $0$ because the corresponding  entries have degree at least two.

Also, by  Theorem \ref{main},  from ${\bf{G.}} $  we can build  the
minimal $R$-free resolution  (${\bf{F.} }, \delta. $)  of $M.  $  By
Theorem \ref{M},  $\beta_j(M)= \beta_j(M^*). $ In  its  turn,  from
(${\bf{F.} },\delta. $)  we can build  $lin^R({\bf{F.}}). $  We have
to prove that
 $$  H_i (lin^P({\bf{G.}}))=0  \ \ \Longrightarrow   \   H_i (lin^R({\bf{F.}}))=0   {\text {\ \ for \ every   }} i \ge 1. $$ We remark that $G_j^{lin} = F_j^{lin} \simeq P^{\beta_j(M)} (-j ) $ where $\beta_j(M)= \beta_j(M^*) = \sum_{r=1}^m \beta_{j, i_r+j} $ and, without loss of generality, we denote by $ e^j_{s,i_r} $ for  $r=1, \dots, m, $   $s=1, \dots, \beta_{j, i_r+j} $ a basis  of both $ G_j^{lin} $ and $ F_j^{lin}.$ The matrices  $ lin (\mathcal{M}_{j})  $  associated  to $ \delta_j^{ lin} : P^{\beta_j(M )} (-j ) \to P^{\beta_{j-1}(M)} (-j+1) $$   $ with respect to these  bases  are obtained from  $ \mathcal{M}_{j}  $  and they have the following shape:

   \[
    \xymatrix@C=   1.5ex@R=0.25ex{
 & &                i_1+j        & &                    & &             i_{m}+j   & && \\
 & & \hspace{7ex}\drop\frm{^\}} &  & & & \hspace{7ex}\drop\frm{^\}} &  &  \\
 & \ar@{-}@/_1pc/[dddddd] & & \ar@{.}[dddddd]  &    &   \ar@{.}[dddddd]  &  &  \ar@{-}@/^1pc/[dddddd]   &  \\
i_1+j-1 \,\bigg\{ & &     M_{i_1 i_1 j}^*           & &     0       & &            0 &  \\
 & \ar@{.}[rrrrrr] &&&&&&& \\
                      & &        ^*               &&            \ddots                &&         0            && \\
  & \ar@{.}[rrrrrr] &&&&&&& \\
 i_m +j-1  \,\bigg\{ & &              ^*     & &            ^*         & &      M_{i_m  i_m  j}^*  & & \\
 &&&&&&&&
}
\]
where   the ``diagonal blocks" coincide with those of $ lin   (\mathcal{M}_{j}^*)  $ (the non zero entries have valuation $1$) and $*$ denotes $0$  or linear forms.
 Since we always have $ Ker\ \delta_j^{lin} \supseteq Im\  \delta_{j+1}^{lin},  $ we prove  $ Ker\  \delta_j^{lin} \subseteq Im\  \delta_{j+1}^{lin}  $
 for every $j=1, \dots, pd(M).$

 Fixed an integer $r \in \{1, \dots, m\}, $  denote by $N^j_{i_r} $ the submodule of $ G_j^{lin} $ generated by $  e^j_{1, i_r}, \dots   e^j_{\beta_{j, i_r+j}, i_r}. $  Since $ {lin^P(\bf{G_{i_r} }})  $ is exact,  then
 $ Ker\  (d_j^{\ lin}|N^j_{i_r}) = Im\  (d_{j+1}^{\ lin}|N^{j+1}_{i_r}). $
Let    $$x = \sum_{r=1}^m \sum_{s=1}^{\beta_{j, i_r+j}} \lambda_{s  r} e^j_{s, i_r}  \in  Ker\  \delta_j^{lin} $$ with $\lambda_{ s  r }  \in  P. $ Because $\sum_{s=1}^{\beta_{j, i_1+j}} \lambda_{ s 1} e^j_{s, i_1} \in Ker\ d_j^{\ lin} \cap N^j_{i_1} = Im\ (d_{j+1}^{\ lin}| N^{j+1}_{i_1}), $  we get    $\sum_{s=1}^{\beta_{j, i_1+j}} \lambda_{ s 1} e^j_{s, i_1}= d_{j+1}^{lin} (\alpha_1) $ with $\alpha_1 \in     N^{j+1}_{i_1}. $
Set  $$x_1:= x -  \delta_{j+1}^{lin} (\alpha_1). $$ Notice that $ x_1 \in  Ker\  \delta_j^{lin}$ because  $x  \in  Ker\  \delta_j^{lin}$ and $    \delta_j^{lin}\circ    \delta_{j+1}^{lin}=0.$  It is easy to see that  $x_1 \in < e^j_{s, i_r} > $ with $r  \ge 2.$ In fact $\delta_{j+1} (\alpha_1)- d_{j+1}(\alpha_1) \in  P^{\beta_j(M)} (-j )/N^j_{i_1}$ because the blocks  on the diagonal of $lin^R({\bf{F.}}) $ and $lin^P({\bf{G.}}) $ coincide.

Hence  $ x_1 = \sum_{r=2}^m \sum_{s=1}^{\beta_{j, i_r+j}} \lambda_{s  r}'  e^j_{s, i_r}  \in  Ker\  \delta_j^{lin} $ with $\lambda_{ s  r }'  \in  P  $ and $ \sum_{s=1}^{\beta_{j, i_2+j}} \lambda_{s  r}'  e^j_{s, i_r} \in
 Ker\ d_j^{\ lin} \cap N^j_{i_2} = Im\ (d_{j+1}^{\ lin}| N^{j+1}_{i_2}). $  We can repeat the same procedure
 and finally we find $\alpha_r \in    N^{j+1}_{i_r}, $  $r=1, \dots, m-1, $ such that
 $$ x_m:= x- \delta_{j+1}^{\ lin}(\sum_{i=1}^{m-1}\alpha_i) \in  Ker\ d_j^{\ lin} \cap N^j_{i_m} = Im\ (d_{j+1}^{\ lin}| N^{j+1}_{i_m})=   Im\ (\delta_{j+1}^{\ lin}| N^{j+1}_{i_m}).$$ Because  $x_m \in  Im\ (\delta_{j+1}^{\ lin}),  $ it follows that $x \in  Im\ (\delta_{j+1}^{\ lin}),  $ as required.

   \end{proof}

 As a consequence of Theorem \ref{lpd}  and of  \cite[Theorem 1.5 and Remark 1.6]{CHH},  we can prove the following result.
  \begin{corollary} \label{alfa} Let $I$ be an ideal of a regular local ring $R$ of dimension $n$ which satisfies one of the equivalent  conditions of Corollary \ref{I*} and let $A=R/I$. Then
    $$\beta_i(A) =  \sum_{j=1}^{n-i+1} {{n-j}\choose {i-1} } \alpha_j(A).$$

 \end{corollary}

\noindent  We remark that, under the   assumption of the above result,  we also have  $$\beta_i(A) =  \sum_{j=1}^{n-i+1} {{n-j}\choose {i-1} } \alpha_j(gr_{\m}(A)).$$
 \vskip 3mm
 We present now an unexpected  consequence on the theory of blowing-up algebras.
 If $I \subseteq \n^2 $ is a non zero ideal of a regular local ring $(R, \n),$ we let $A=R/I $ the local ring with maximal ideal $\m=\n/I. $ We denote by $S_A(\m) $ the symmetric algebra of $\m$ over $A.$  In this case it is known (see \cite[Corollary 2.2]{HRV}) that $$S_A(\m) = \oplus_{p \ge 0} \n^p/ I \n^{p-1}. $$  Even if $A$ is Cohen-Macaulay, the symmetric algebra has a strong tendency not to be Cohen-Macaulay.  M.E. Rossi (see \cite[Theorem 2.4 and Theorem 3.3]{Ro}) proved that  $S_A(\m) $ is Cohen-Macaulay if and only if $A$ is an abstract hypersurface ring. One reason is that the  Krull dimension of $S_A(\m),  $ compared with the one of $A$, can be very higher. C. Huneke and M.E. Rossi (see \cite{HR}) gave an explicit formula for the dimension of the symmetric algebra. Applied to the above setting, we  get

 \begin{equation} \label{dimension}\dim S_A(\m)= \dim R.
 \end{equation}    It follows from a result of the same paper \cite{HR} that
 $$ \depth ~ S_A(\m) \le \dim ~A +1.$$
In \cite[Theorem 2.13]{HRV}, J. Herzog, M.E. Rossi and G. Valla  proved   that $ \depth~  S_A(\m) $  (with respect to homogeneous irrelevant maximal ideal) is strictly related to the  depth  of $gr_{\n}(I).$
As a consequence, by using this connection and  the results of this paper, we prove the following theorem.

 \begin{theorem} \label{sym}
Let  $I \subseteq \n^2 $ be an ideal of a regular local ring $(R, \n),$ we let $A=R/I $ the local ring  with maximal ideal $\m=\n/I. $  Assume that $ I$ satisfies one of the equivalent  conditions of Corollary \ref{I*}. If $\depth ~ A>0, $ then
 $$\depth ~S_A(\m) \ge \depth~ A +1.$$
 If $A$ is Cohen-Macaulay, then $ \depth~S_A(\m) = \dim~ A +1.$
 \end{theorem}
 \begin{proof} By Theorem  \ref{lpd} we know that   $I$ is of homogeneous type. It follows that  $\depth~ gr_{\n}(I) = \depth ~I =  \depth~ A +1 $ and the result follows by \cite[Theorem 2.13 c) and  Corollary 4.14]{HRV}.
 \end{proof}

\noindent  We remark that, if $\depth ~A=0, $ then $\depth~ S_A(\m) =0 $ from \cite[Proposition 2.3 d)]{HRV}.
 \\

Notice that Theorem \ref{sym}  extends and reproves  \cite[Theorem
3.9]{HRR} which was showed   in the homogeneous  context.  Moreover,
Theorem \ref{sym} and equality (\ref{dimension}) show that if $A$ is
Cohen-Macaulay and the equivalent conditions of Corollary \ref{I*}
hold, then $ S_A(\m) $ is Cohen-Macaulay if and only if $ \dim~ A=
\dim R - 1, $ which means $A$ is a hypersurface ring. The result
recovers, in a particular case, a more general result proved by M.E.
Rossi  in \cite[Theorem 3.3]{Ro}.


\begin{thebibliography}{100}

\bibitem[AHH]{AHH}A. Aramova, J. Herzog and T. Hibi, Ideals with stable Betti numbers, Adv. Math. 152 (2000),
no. 1, 72--77.

\bibitem[AE]{AE} L.  Avramov, D. Eisenbud,  Regularity of modules over a Koszul algebra, J. Algebra 153
(1992), no. 1, 85--90.

\bibitem[Bi]{Bi}  A. M. Bigatti, Upper bounds for the Betti numbers of a given Hilbert function, Comm. Algebra 21 (7) (1993) 2317--2334.

\bibitem[B]{B} V. Bertella, Hilbert function of local Artinian level rings in codimension two, J. Algebra  321 (2009), no. 5, 1429--1442.

 \bibitem[BH]{BH} W. Bruns, J. Herzog, {\em Cohen-Macaulay rings.} Cambridge Studies in Advanced Mathematics, 39. Cambridge University Press, Cambridge, 1993.


  \bibitem[CoCoA]{CoCoA} A. Capani, G. De Dominicis, G. Niesi, L. Robbiano,  CoCoA: a system for doing Computations in Commutative Algebra,  Available at http://cocoa.dima.unige.it.

\bibitem[C]{C} A. Conca, { Koszul homology and extremal properties of Gin and Lex}, Trans. Amer. Math. Soc. 356 (2004), no. 7, 2945--2961.

\bibitem[CHH]{CHH} A. Conca, J.Herzog, T.Hibi, { Rigid resolutions and big Betti numbers}, Comment. Math. Helv. 79 (2004), no. 4, 826--839.


\bibitem[E]{E} D. Eisenbud, {\em Commutative algebra. With a view toward algebraic geometry}. Graduate Texts in Mathematics, 150. Springer-Verlag, New York, 1995.

\bibitem[EFS]{EFS}  D. Eisenbud,  G. Floystad and F. O. Schreyer, Sheaf cohomology and free resolutions over exterior algebras. Trans. Amer. Math. Soc. 355 (2003), no. 11, 4397--4426.

\bibitem[EV]{EV}  J. Elias, G. Valla,    Structure theorems for certain Gorenstein ideals,  Michigan Journal
of Math. 57 (2008), 269--292..

\bibitem[EV1]{EV1} J. Elias,  G. Valla,   Rigid Hilbert functions. J. Pure Appl. Algebra 71 (1991), no. 1, 19--41.


\bibitem[HH]{HH} J. Herzog, T. Hibi, {  Componentwise linear ideals}, Nagoya Math. J. 153, (1999), 141--153.

\bibitem[H]{H}  J. Herzog, The linear strand of a graded free resolution, unpublished notes (1998).

\bibitem[HRR]{HRR}   J. Herzog,  G. Restuccia,  G. Rinaldo,  {  On the depth and regularity of the symmetric algebra}. BeitrŠge Algebra Geom. 47 (2006), no. 1, 29--51.

\bibitem[HRW]{HRW}   J. Herzog, V. Reiner and V. Welker, Componentwise linear ideals and Golod rings, Michigan Math. J. 46 (1999), no. 2, 211--223.

\bibitem[HI]{HI} J. Herzog, S.  Iyengar,  Koszul modules. J. Pure Appl. Algebra 201 (2005), no. 1-3, 154--188.

\bibitem[HRV]{HRV} J. Herzog, M. E. Rossi, G. Valla,   On the depth of the symmetric algebra. Trans. Amer. Math. Soc. 296 (1986), no. 2, 577--606.

\bibitem[HSV]{HSV} J. Herzog, A. Simis and W. Vasconcelos,  Koszul homology and blowing-up rings, Commutative Algebra   (Proc. Trento Conference), Dekker, New York, 1983, 79--169.

 \bibitem[Hu]{Hu} H. A. Hulett,  Maximum Betti numbers of homogeneous ideals with a given Hilbert function. Comm. Algebra 21 (1993), no. 7, 2335--2350.

 \bibitem[HR]{HR} C. Huneke,  M. E. Rossi, The dimension and components of symmetric algebras. J. Algebra 98 (1986), no. 1, 200--210.

 \bibitem[KR]{KR} M. Kreuzer, L.Robbiano, {\em Computational commutative algebra 1. }, Springer-Verlag, Berlin, 2000.
 
 \bibitem[MZ]{MZ} R. Martnez-Villa, D. Zacharia, Approximations with modules having linear resolutions.,
J. Algebra 266 (2003), 671--697.

 \bibitem[P]{P}  K. Pardue, Deformation classes of graded modules and maximal Betti numbers,  Illinois J. Math. 40 (1996), no. 4, 564--585.

\bibitem[Rob]{Rob} L. Robbiano, Coni tangenti a singolarita' razionali, Curve algebriche, Istituto di Analisi Globale, Firenze, 1981.


\bibitem[RoV]{RoV} L. Robbiano, G. Valla, Free resolutions for special tangent cones, Commutative
algebra (Trento, 1981), 253--274, Lecture Notes in Pure and Appl.
Math. 84, Dekker, New York, (1983).

\bibitem[RoV1]{RoV1}  L. Robbiano, G. Valla,     On the equations defining tangent cones. Math. Proc. Cambridge Philos. Soc. 88 (1980), no. 2, 281--297.

\bibitem[R]{R} T. R\"oemer, On minimal graded free resolutions, Dissertation zur Erlangung des
Doktorgrades (Dr. rer. nat.),  Univ. Essen (2001).

 \bibitem[Ro]{Ro} M. E. Rossi,  On symmetric algebras which are Cohen Macaulay. Manuscripta Math. 34 (1981), no. 2-3, 199--210.

  \bibitem[RS]{RS} M. E. Rossi, L. Sharifan, Consecutive cancellations in Betti numbers of local rings, Proc. A.M.S. (to appear)

\bibitem[RV]{RV} M. E. Rossi, G. Valla, Hilbert Function of filtered modules, (2007)  arXiv:0710.2346.

\bibitem[Sa]{Sa}   J. D. Sally, Stretched Gorenstein rings, J. London Math. Soc. 20 (1979), no. 2,
19--26.

\bibitem[Se]{Se}  L. M. Sega, Homological properties of powers of the maximal ideal of a local ring, J. Algebra 241 (2001)
827--858.

\bibitem[Sh]{Sh} T. Shibuta, Cohen-Macaulyness of almost complete intersection tangent cones,  J.  Algebra 319 (2008), no. 8, 3222--3243.

\bibitem[VV]{VV} P. Valabrega , G. Valla,  Form rings and regular sequences, Nagoya Math. J. 72(2) (1978), 475--481.

 

\end{thebibliography}
\end{document}